\documentclass[11pt]{amsart}
\usepackage[a4paper,hmargin=3cm,vmargin=3cm]{geometry}
\usepackage{amsfonts,amssymb,amscd,amstext,amsthm}
\usepackage[pdftex]{graphicx}
\usepackage[dvips]{epsfig}
\usepackage[T1]{fontenc}
\usepackage[english]{babel}
\usepackage[utf8]{inputenc}
\usepackage[bookmarksnumbered,plainpages]{hyperref}
\usepackage{indentfirst}
\usepackage{color}
\usepackage{url}

\usepackage[shortlabels]{enumitem}
\usepackage{imakeidx}
\usepackage{titlesec}
\usepackage{mathrsfs}
\usepackage{stmaryrd}
\usepackage{textfit}
\usepackage{lipsum}
\usepackage{amsmath}
\usepackage{mathtools}

\usepackage{lineno}

\newtheorem*{conj}{Conjecture}

\usepackage{fancyhdr}
\usepackage{times}

\def \R{\mbox{${\mathbb R}$}}

\def \H{\mbox{${\mathbb H}$}}

\def \Q{\mbox{${\mathbb Q}$}}

\newcommand{\rank}{\text{rank}}

\newcommand{\Ima}{\text{Im}}

\titleformat{\section}
{\filcenter\bfseries\large} {\thesection{.}}{0.2cm}{}
\titleformat{\subsection}[runin]
{\bfseries} {\thesubsection{.}}{0.15cm}{}[.]
\titleformat{\subsubsection}[runin]
{\em}{\thesubsubsection{.}}{0.15cm}{}[.]

\usepackage[up,bf]{caption}

\newtheorem{theorem}{Theorem}[section]
\newtheorem{lemma}[theorem]{Lemma}
\newtheorem{proposition}[theorem]{Proposition}
\newtheorem*{maintheorem}{Main Theorem}
\newtheorem{remark}[theorem]{Remark}

\theoremstyle{definition}
 
\numberwithin{equation}{section}
\numberwithin{figure}{section}

\begin{document}
    
\fancyhead[LO]{On minimal homogeneous submanifolds of \(\H^{n+2}\)}
\fancyhead[RE]{Felippe Guimar\~aes, Joeri Van der Veken}
\fancyhead[RO,LE]{\thepage}

\thispagestyle{empty}

\begin{center}
{\bf \LARGE On minimal homogeneous submanifolds of the hyperbolic space up to codimension two}
\vspace*{5mm}

\hspace{0.2cm} {\Large Felippe Guimar\~aes and Joeri Van der Veken} 
\end{center}


\begin{quote}
{\small
\noindent {\bf Abstract}\hspace*{0.1cm}
    We show that a minimal homogeneous submanifold $M^n$, $n\geq 5$, of a hyperbolic space up to codimension two is totally geodesic.
}
\\

{\small
\noindent {\textbf{Mathematics Subject Classification:}}\hspace*{0.1cm}
    53C42, 53C40, 53C30.
}
\\
{\small
\noindent {\textbf{Keywords:}}\hspace*{0.1cm}
    Isometric immersion, extrinsically homogeneous submanifolds, minimal submanifolds.
}

\end{quote}

\section{Introduction}
We say that a Riemannian manifold \( M^n \) is \textsl{homogeneous} if the Lie group \( \mathrm{Iso} (M^n) \) of all isometries of \( M^n \) acts transitively on \( M^n \). Let \( f: M^n \rightarrow \Q_c^{n+p} \) be an isometric immersion of a homogeneous \( n \)-dimensional Riemannian manifold \( M^n \) into a simply connected \( (n+p) \)-dimensional Riemannian manifold with constant sectional curvature \( c \). The latter is known as a space form. We say that the isometric immersion \( f \) is (isometrically) rigid if it is unique up to isometries of \( \Q_c^{n+p} \). In particular, rigid homogeneous submanifolds are orbits of an action of some subgroup of \( \mathrm{Iso} (\Q_c^{n+p}) \); such manifolds are called \textsl{extrinsically homogeneous} submanifolds.

The classification of homogeneous hypersurfaces of simply connected space forms was completed in the works \cite{NaganoTakahashiHomEuclidean,TakahashiCod1Hom,TakahashiCod1Hom3,TakagiTakahashiHomSphere}. An alternative way to obtain such a classification is to study when the hypersurface is extrinsically homogeneous, and this is established by studying its rigidity using \cite{Killing}. More precisely, if the rank of the second fundamental form of the isometric immersion is greater than two, then it is rigid, and if the submanifold is homogeneous, then it will be an isoparametric hypersurface, i.e., the parallel equidistant hypersurfaces have constant mean curvature. In this case, the result follows from the fact that the homogeneous isoparametric hypersurfaces of simply connected space forms were classified in \cite{Somigliana,Segre,CartanIsoHyp,HsiangLawsonMinimalHomogeneous}. The cases in which the hypersurface is not rigid (a priori) are those in which the second fundamental form is highly degenerate and are treated separately.

The codimension two case was considered in \cite{spaceFormsNoronhadeCastro,euclideanNoronhadeCastro} under an additional hypothesis regarding the rank of the second fundamental form. A hypotheses is required to use results on isometric rigidity in \cite{doCarmoDajczerRigidity} (see also \cite{base}). In particular, they classified compact homogeneous submanifolds of codimension two.

Concerning extrinsically homogeneous submanifolds, there are interesting results that do not depend on the codimension. In \cite{homHypdiScalaOlmos,homMinEuclideanDiScala}, the authors proved that the only minimal extrinsically homogeneous submanifolds of Euclidean space and of hyperbolic space are the totally geodesic ones. In \cite{DiScalaKahlerMinimal}, the author proposed the following conjecture. 

\begin{conj}[\cite{DiScalaKahlerMinimal}]
Let \( M^n \) be a Riemannian manifold that is either locally homogeneous or Einstein. Then any minimal isometric immersion \( f: M^n \rightarrow \Q_c^{n+p} \), \( c \leq 0 \), must be totally geodesic.
\end{conj}

The original statement did not include the hyperbolic case \( c < 0 \). In \cite{MatsuyamaEinsteinMinCod2} (also discussed in \cite{GuilhermeEinsteinMinCod2}), the conjecture was confirmed for Einstein submanifolds of codimension two. When the submanifold is a surface with constant Gaussian curvature, the conjecture follows from \cite[Theorem 4.2]{BryantCMCsufraces}. However, the scenario where the submanifold is homogeneous, with \( n \neq 2 \) and in codimension two, remains open. This work addresses the conjecture in the context of hyperbolic ambient spaces. If we denote by $\H^{n+2}$ the hyperbolic space of dimension $n+2$ and constant sectional curvature $-1$, the main result is the following.

\begin{maintheorem}\label{thm:mainTheo}
Let \( f: M^n \rightarrow \H^{n+2} \), \( n \geq 5 \), be a minimal isometric immersion of a homogeneous manifold \( M^n \). Then \( f \) is totally geodesic.
\end{maintheorem}

In order to prove the stated theorem, we will first show that the homogeneous submanifold cannot have a positive index of relative nullity (dimension of the kernel of the second fundamental form) everywhere. We will use the splitting tensor (also known as the conullity operator), a tool previously used in other works to identify constraints on the existence of complete submanifolds with a positive index of relative nullity (cf.  \cite{trabBoys,RosenthalSplittingTensor,base}). Once we have established that the second fundamental form is non-degenerate, the approach from \cite{spaceFormsNoronhadeCastro} will be applied to transform the problem into an algebraic one, focusing on verifying when the fundamental equations of the submanifold are satisfied.

The paper is organized as follows. In Section~\ref{sec:Pre} we provide definitions and results that will be used throughout the work. In Section~\ref{sec:Proof} we give the proof of the Main Theorem.


\section{Preliminaries}\label{sec:Pre}
Let \( f: M^n \rightarrow \Q^{n+p}_c \) be an isometric immersion of a Riemannian manifold. We denote by \( \alpha(x) \), the second fundamental form, and by \( A_\xi(x) \), the shape operator of \( f \) with respect to a normal vector \( \xi \in T^\perp_f M(x) \) at \( x \in M^n \). We say that $f$ is \emph{isometrically rigid}, or simply \emph{rigid}, if any other isometric immersion $g: M^n \rightarrow \Q^{n+p}_c$ is congruent to it by an isometry of the ambient space $\Q^{n+p}_c$. That is, there exists an isometry $\mathcal{I} \in \mathrm{Iso} \left( \Q^{n+p}_c\right)$ such that $g = \mathcal{I} \circ f$.

The \emph{relative nullity subspace} $\Delta\left(x\right) \subset T_xM$ of $f$ at $x \in M^n$ is the kernel of its second fundamental form $\alpha$ at $x \in M^n$, specifically,
\begin{equation*}
    \Delta \left(x\right)=\left\{X\in T_xM:\alpha\left(X,Y\right)=0\text{ for all }Y\in T_xM\right\}.
\end{equation*}
The dimension of this subspace, $\nu(x) := \dim \Delta(x)$, is referred to as the \emph{index of relative nullity} at $x \in M^n$. Define $\nu_0 := \min_{x \in M} \nu(x)$ as the minimum value of the relative nullity. The isometric immersion $f$ is considered \emph{totally geodesic} at $x \in M^n$ if $\nu(x) = n$, and $f$ is a \emph{totally geodesic} isometric immersion if $\nu \equiv n$ throughout $M^n$.

For a given $c \in \mathbb{R}$, the \emph{$c$-nullity distribution} of $M^n$ is the variable rank (intrinsic) distribution $\Gamma_c$ on $M^n$, defined at each $x \in M$ by
\[ 
\Gamma_c(x) = \left\{ Z \in T_xM : R(X, Y)Z = c\left( \langle Y, Z \rangle X - \langle X, Z \rangle Y \right) \text{ for all } X, Y \in T_xM \right\}.
\]
The \emph{index of $c$-nullity} at $x \in M^n$ is denoted by $\mu_c(x) := \dim \Gamma_c(x)$. Here we consider the curvature tensor $R$ of the Levi-Civit\`a connection $\nabla$, adhering to the sign convention 
\[
R(X,Y)Z=\nabla_X\nabla_YZ-\nabla_Y\nabla_XZ-\nabla_{[X,Y]}Z,
\]
for vector fields $X$, $Y$, $Z\in\Gamma(TM)$.

It is known that $\Delta$ and $\Gamma_c$ are autoparallel distributions on any open set where their dimensions are locally constant. Since $\nu$ is a lower semicontinuous function, the set $U = \{x \in M^n : \nu(x) = \nu_0\}$ is open and, in case $M^n$ is a complete Riemannian manifold, the leaves of the relative nullity in $U$ are complete totally geodesic submanifolds of constant curvature $c$ (cf.  \cite{FerusTotGeoFoliations,GrayConullity,MaltzNullity}).

Note that $\Delta(x) \subset \Gamma_{c}(x)$ and $\nu(x) \leq \mu_{c}(x)$ by the Gauss equation. For minimal submanifolds of a simply connected space form $\Q^{n+p}_c$, we have the equality.

\begin{lemma}\label{lem:minEq}
Let $f: M^n \rightarrow \Q^{n+p}_c$ be a minimal isometric immersion. Then $\Gamma_c(x) = \Delta(x)$ for all $x \in M^n$.
\end{lemma}
\begin{proof}
Let $x\in M^n$. The inclusion $\Delta(x) \subset \Gamma_c(x)$ follows in general from the Gauss equation as stated above. We now prove the other inclusion using minimality. Let $Y \in \Gamma_c(x)$. It follows from the Gauss equation that
\[
0 = \langle \alpha(Y,Y), \alpha(X,X)\rangle - \Vert \alpha(X,Y) \Vert^2
\]
for all $X \in T_xM$. Summing over an orthonormal basis $\{X_i\}_{1 \leq i \leq n}$ of $T_xM$ and using the minimality of the isometric immersion, we have $\sum_{1 \leq i \leq n}\Vert \alpha(X_i,Y) \Vert^2 = 0$ and, in particular, $Y \in \Delta(x)$.
\end{proof}

We say that $M^n$ is an \emph{extrinsically homogeneous submanifold} of $\Q^{n+p}_c$ if for any two points $x,y \in M^n$ there exists an isometry $\mathcal{I} \in \mathrm{Iso} (\Q^{n+p}_c)$ such that $\mathcal{I}(M) = M$ and $\mathcal{I}(x)=y$. We can now state an important result that will play a key role in the proof of our result.

\begin{theorem}[Corollary 1.4 in \cite{homHypdiScalaOlmos}]\label{cor:homScalaOlmos}
A minimal extrinsically homogeneous submanifold of hyperbolic space must be totally geodesic.
\end{theorem}

It follows from the definition of isometric rigidity that a rigid homogeneous submanifold in $\Q_c^{n+p}$ is an extrinsically homogeneous submanifold. Thus, a natural approach to the proof of the Main Theorem is to study the cases when the submanifold is not rigid in the following result.

\begin{theorem}[Theorem 1 in \cite{spaceFormsNoronhadeCastro}]\label{thm:HypRank2}
Let $f: M^n \rightarrow \Q_c^{n+2}$, $c \in \R$ and $n \geq 5$, be an isometric immersion of a Riemannian homogeneous manifold such that $\nu_0 \leq n-5$. Then either $f$ is rigid, or for every point $x \in M^n$ there exist orthonormal vectors $\xi, \eta \in T^\perp_f M(x)$ such that $\rank\ A_\eta \leq 2$ and $i_* \circ A_\xi = A_\xi \circ i_*$ for all $i \in \mathrm{Iso} (M^n)$.

\end{theorem}

Below, we restate a result from \cite{spaceFormsNoronhadeCastro} indicating that, even when the submanifold is not rigid, it exhibits a special structure.

\begin{theorem}[Theorem 19 in \cite{spaceFormsNoronhadeCastro}]\label{thm:TotallyGeodesicProduct}
Let \( f: M^n \rightarrow \H^{n+2} \), \( n \geq 5 \), be an isometric immersion of a Riemannian homogeneous manifold with $\nu_0 \leq n-5$, such that for each \( x \in M^n \) there exists an orthonormal frame \(\{ \xi, \eta \}\) of the normal space with \( A_\xi \) invariant by isometries of \( M^n \) and \( \rank\ A_\eta = 2 \). Then one of the following occurs:
\begin{itemize}
    \item \( \ker A_\eta \) is an autoparallel distribution and \( M^n \) is a product \( \Sigma^2 \times \Q^{n-2}_c \) or \( \Sigma^3 \times \mathbb Q^{n-3}_c\), where in the latter case \( c < 0 \);
    \item \( \ker A_\eta \) is not an autoparallel distribution and \( M^n \) is a cohomogeneity one manifold such that all orbits are flat spaces.
\end{itemize}
\end{theorem}

Therefore, it only remains for us to study the case in which the submanifold has a minimum nullity index $\nu_0 > n-5$. In order to work with this case, we will introduce a tensor widely used to work with such a distribution. Consider the orthogonal splitting $TM=\Delta\oplus\Delta^\perp$. For a vector field $X\in\Gamma(TM)$, we shall write $X=X|_{\Delta}+X|_{\Delta^\perp}$. Now we can define the 
\emph{splitting tensor (or conullity operator)} of $\Delta$ as the map
\[ C:\Gamma(\Delta)\times\Gamma(\Delta^\perp)\to \Gamma(\Delta^\perp) \]
given by
\[ C_T X \coloneqq C(T,X)=-(\nabla_XT)|_{\Delta^\perp}.\]
It is clear that $C$ is $C^\infty(M)$-linear in each variable. Such a tensor was introduced in \cite{RosenthalSplittingTensor} and is discussed in more detail in \cite[\textsection 7.2]{base}. 

Let $\gamma: I \subset \mathbb{R} \rightarrow M^n$ be a geodesic inside a leaf of the nullity distribution $\Delta$, the splitting tensor associated with the distribution $\Delta$ satisfies a Riccati's type equation
\[ \frac{D}{dt}C_{\gamma'} = C_{\gamma'}^2 - I.\]
We can explicitly integrate the previous equation to get the following proposition (cf.  \cite{trabBoys}).

\begin{proposition}[Proposition 1 in \cite{trabBoys}]\label{prop:eqn-c}
Let $f: M^n \rightarrow \H^{n+p}$ be an isometric immersion and $\gamma:[0,b) \to M^n$ a nontrivial unit speed geodesic with $p=\gamma(0)$ and 
$\gamma'(0)\in\Delta(p)$ so that $\gamma$ is a geodesic of the 
leaf of $\Delta$ through~$p$. Assume that $\gamma([0,b))$
is contained in an open subset of~$M^n$ where $\nu\neq0$ is constant. 
Then the splitting tensor $C(t) \coloneqq C_{\gamma'(t)}$ of $\Delta$ at $\gamma(t)$ 
is given by 
\begin{equation}\label{eq:splitting}
	C(t)=-\mathcal{P}^t_0 \circ \left(J_0'(t) \circ J_0(t)^{-1} \right) \circ \left(\mathcal{P}^t_0\right)^{-1},
\end{equation}
where
\begin{equation}\label{eq:defJacobi}
J_0(t) =
(\cosh t) I - (\sinh t) C_0,
\end{equation} 
$C_0 \coloneqq C(0)$ and $\mathcal{P}^t_0$ is the parallel transport along $\gamma$. Moreover, the second fundamental form is given by
\begin{equation}\label{eq:secFund}
\alpha_{\gamma(t)} = \mathcal{P}^t_0 \circ \alpha_{\gamma(0)}\left(J_0(t)^{-1} \circ \left(\mathcal{P}^t_0\right)^{-1},\left(\mathcal{P}^t_0\right)^{-1}\right).
\end{equation}
\noindent In particular, if the geodesic is defined for all $\R$ then any real eigenvalue~$\lambda$ of $C_{\gamma'}$ satisfies $\vert\lambda\vert\leq 1$.
\end{proposition}

\section{Proof of the Main Theorem}\label{sec:Proof}
Let $f: M^n \to \mathbb H^{n+2}$ be a minimal isometric immersion of a homogeneous manifold into a hyperbolic space. We will consider two cases: first, we assume that the submanifold has nontrivial relative nullity, and then we assume that the index of relative nullity is zero.

\subsection{The submanifold has nontrivial relative nullity}
Minimal homogeneous submanifolds of a hyperbolic space for which the relative nullity is everywhere nontrivial must be totally geodesic, regardless of the codimension. In fact, we have the following more general result.

\begin{proposition}\label{prop:scalBound}
Let $f: M^n \rightarrow \H^{n+p}$ be an isometric immersion of a complete manifold with constant scalar curvature, mean curvature vector field of constant length, and whose relative nullity is everywhere nontrivial. Then $f$ is totally geodesic.
\end{proposition}
\begin{proof}
Let $x \in U \subset M^n$, where $U = \{x \in M^n: \nu(x) = \nu_0\}$ is the open subset where the index of relative nullity attains its minimum. By \cite{FerusTotGeoFoliations}, the leaves of relative nullity within \( U \) are complete. Let $\gamma: \mathbb{R} \rightarrow U \subset M^n$ be a complete geodesic such that $\gamma(0) = x$, $\gamma'(0) \in \Delta(\gamma(0))$, and $\Vert \gamma' \Vert = 1$. From Proposition~\ref{prop:eqn-c}, we have that $C(t) \coloneqq C_{\gamma'(t)} = -\mathcal{P}^t_0 \circ \left(J_0'(t)J_0(t)^{-1} \right) \circ \left(\mathcal{P}^t_0\right)^{-1}$, where $J_0(t) = \cosh t I - \sinh t C_0$ is invertible for all $t \in \mathbb{R}$.

If there is $\lambda \in \mathbb{R}$ such that $|\lambda| = 1$ and $C(0)X_0 = \lambda X_0$ for some unit vector $X_0 \in \Delta^\perp(\gamma(0))$, then $$J_0^{-1}(t)X_0 = \frac{2}{(1-\lambda)e^t+(1+\lambda)e^{-t}} X_0,$$ and $\Vert J_0^{-1}(t) X_0 \Vert$ is unbounded when $t$ goes to $\infty$ (if $\lambda=1$) or $-\infty$ (if $\lambda=-1$). Moreover, by \eqref{eq:secFund}, we have that $\Vert \alpha_{\gamma(t)} (\mathcal{P}^t_0 X_0, \mathcal{P}^t_0 Y_0) \Vert$ goes to infinity when $t$ goes to $\infty$ or $-\infty$ for some $Y_0 \in \Delta(\gamma(0))^{\perp}$ (otherwise $X_0 \in \Delta(\gamma(0))$). In particular, $\Vert \alpha \Vert^2$ is not bounded, and this is a contradiction with the expression of the scalar curvature $s$ given by the Gauss equation
\begin{equation}\label{eq:scalarCurvature}
    s = -1 + \frac{n}{n-1} \Vert \mathcal{H} \Vert^2 - \frac{1}{n(n-1)}\Vert \alpha \Vert^2. 
\end{equation}

We can now assume that the matrix $(I - C_0)$ is invertible. We have that $\left[J_0(t)\right]^{-1} = \frac{1}{\cosh t} \left( I - \tanh t C_0  \right)^{-1}$ exists for every $t\in \mathbb{R}$, and since $\left( I - C_0  \right)$ is invertible, it follows that $\left[J_0(t)\right]^{-1}$ converges to the zero matrix when $t$ goes to $\infty$. It follows from \eqref{eq:secFund} that $\Vert \alpha \Vert^2$ goes to zero along $\gamma$, using \eqref{eq:scalarCurvature} and the hypothesis that the scalar curvature $s$ and $\Vert \mathcal{H} \Vert$ are constants, we have that $\Vert \alpha \Vert^2$ must be zero everywhere, i.e., $f$ is totally geodesic.
\end{proof}

\subsection{The submanifold has trivial relative nullity}
In this case, we are under the assumptions of Theorem~\ref{thm:HypRank2}. If $f$ is isometrically rigid, then by Theorem~\ref{cor:homScalaOlmos}, the submanifold is totally geodesic. Thus, we can assume that $f$ is not rigid. For every point $x \in M^n$, there exist orthonormal vectors $\xi, \eta \in T^\perp_fM(x)$ such that $\rank\ A_\eta \leq 2$ and $A_\xi$ is invariant under $\mathrm{Iso} (M^n)$ by Theorem \ref{thm:HypRank2}. The next lemma establishes that, with the properties of homogeneity and minimality, $\rank\ A_\eta$ can be assumed to be constant equal to \(2\) everywhere.

\begin{lemma}\label{lem:Rank2}
    Let \( f: M^n \rightarrow \mathbb{H}^{n+2} \) be a minimal isometric immersion of a Riemannian homogeneous manifold. Assume that for each point $x\in M^n$, there exist orthonormal vectors $\xi, \eta \in T_f^\perp M(x)$ such that
    \begin{itemize}
        \item \( \rank\ A_\eta \leq 2 \), and
        \item for any isometry $i \in \mathrm{Iso} (M^n)$, the vector $\xi$ can be oriented such that $i_* \circ A_\xi = A_\xi \circ i_*$.
    \end{itemize}
    Under these conditions, \(\det (A_\eta|_{\text{Im} A_\eta})\) is constant, and either \( f \) is totally geodesic or \( \rank\ A_\eta = 2 \) everywhere. If the latter holds, then the distribution \( \ker A_\eta \) is invariant under the isometries of \( M^n \), involutive, and its leaves are homogeneous manifolds.

\end{lemma}
\begin{proof}
Since \( A_\xi \) is invariant under the isometries of \( M^n \) and \( \rank\ A_\eta \leq 2 \), we deduce from the Gauss equation and the homogeneity of \( M^n \), given that the scalar curvature of \( M^n \) is constant, that \( \det (A_\eta |_{\text{Im}A_\eta}) \) is constant. Hence, either \( \rank\ A_\eta \leq 1 \) everywhere or \( \rank\ A_\eta = 2 \) everywhere.

In the case where \( \rank\ A_\eta \leq 1 \), the minimality of \( f \) implies \( A_\eta \equiv 0 \). Specifically, the curvature tensor \( R^\perp \) of the normal bundle \( T^\perp_fM \) is flat. Then, from the main result in \cite{DajczerReductionCod} (or Proposition 2.4 in \cite{base}), the first normal bundle \( \mathcal{N}^f_1 \coloneqq \text{span}\{\alpha(X,Y): X,Y \in TM\} \) of \( f \) is parallel, implying that \( f \) is contained in a totally geodesic hypersurface of \( \H^{n+2} \). In this scenario, it follows from \cite{TakahashiCod1Hom} that \(f\) must be extrinsically homogeneous. Further, due to its minimality, \( f \) is totally geodesic by Theorem~\ref{cor:homScalaOlmos}.

We now turn to the case where \( \rank\ A_\eta \) is constantly equal to 2. The Gauss equation shows that \( \ker A_\eta \) is invariant under \( \mathrm{Iso} (M^n) \). According to \cite[Lemma 7]{spaceFormsNoronhadeCastro}, the distribution \( \ker A_\eta \) is involutive, with its leaves being homogeneous manifolds. 
\end{proof}

To understand the behavior of the distribution \(\ker A_\eta\) when it is not autoparallel, we will need the following well-known result that follows from the Gauss equation.
\begin{lemma} \label{lem:SplittingProduct}
Let $f: M_1^{n_1} \times M_2^{n_2} \rightarrow \Q_c^{m}$ be a minimal isometric immersion, where $c \leq 0$ and $n_1, n_2 \geq 2$. Then $c = 0$ and $f$ is an extrinsic product of minimal isometric immersions.
\end{lemma}

We are in a position to state the structural result about the distribution $\ker A_\eta$.

\begin{proposition}\label{prop:StructureKerAeta}
    Let $f: M^n \rightarrow \mathbb{H}^{n+2}$ be a minimal isometric immersion of a Riemannian homogeneous manifold. Assume that for each point $x\in M^n$, there exist orthonormal vectors $\xi, \eta \in T_f^\perp M(x)$ satisfying
    \begin{itemize}
        \item $\rank\  A_\eta = 2$, and
        \item for any isometry $i \in \mathrm{Iso} (M^n)$, the vector $\xi$ can be oriented such that $i_* \circ A_\xi = A_\xi \circ i_*$.
    \end{itemize}
    Then the following properties hold.
    \begin{enumerate}
        \item\label{it:NotTotGeo} $\ker A_\eta$ is not an autoparallel distribution.
        \item\label{it:ParallelAlong} $\xi$ and $\eta$ are parallel along $\ker A_\eta$.
        \item\label{it:Eigenvalue} span$\left\{ \nabla_X Y : X,Y \in \ker A_\eta \right\}$ is the $(n-1)$-dimensional eigenspace $E_\lambda$ of $A_\xi$ associated with a non-zero eigenvalue $\lambda$.
        \item\label{it:PerpTotGeo} The orthogonal complement $\ker {A_\eta}^{\perp}$ is an autoparallel distribution.
        \item\label{it:Spherical} $\ker A_\eta$ is a spherical distribution in $M^n$.
    \end{enumerate}
\end{proposition}

\begin{proof}
For item~\textit{(\ref{it:NotTotGeo})}, using Lemma~\ref{lem:Rank2}, we know that $\ker A_\eta$ is invariant under the isometries of $M^n$. Hence, if one of its leaves is totally geodesic, then all the leaves are totally geodesic. In such a case, by Theorem~\ref{thm:TotallyGeodesicProduct}, $M^n$ must be a product of either \(\Sigma^2 \times \Q^{n-2}_c \) or \( \Sigma^3 \times \mathbb Q^{n-3}_c\), where in the latter case \( c < 0 \). The item then follows from Lemma \ref{lem:SplittingProduct}.

\vspace{0.1cm}

For item~\textit{(\ref{it:ParallelAlong})}, we follow a similar approach to the one in \cite[Lemma 6]{DajczerGromollIsoCod2}. Let $h \in \mathrm{Iso} (M^n)$ such that $f \circ h$ is not congruent to $f$. This implies $A_\xi h_* = h_* A_{\xi}$ and $A_\eta h_* \neq \pm h_* A_{\eta}$. Given that the distribution $\ker A_\eta$ is both involutive and invariant under the isometries of $M^n$, it follows that $h_*(\ker A_\eta) = \ker \left( A_{\eta}h_* \right)$.

Considering the Codazzi equation for $f$, we have
$$\nabla_v A_\xi w - A_{\nabla_v^\perp\xi}w - A_\xi \nabla_v w = \nabla_w A_\xi v - A_{\nabla_w^\perp\xi}v - A_\xi \nabla_w v,$$
where $v,w \in TM$. The Codazzi equation for $f \circ h$ is
$$\nabla_{\tilde{v}} A_{\xi} \tilde{w} - A_{\nabla_{\tilde{v}}^\perp \xi }\tilde{w} - A_{\xi} \nabla_{\tilde{v}} \tilde{w} = \nabla_{\tilde{w}} A_{\xi} \tilde{v} - A_{\nabla_{\tilde{w}}^\perp \xi}\tilde{v} - A_{\xi} \nabla_{\tilde{w}} \tilde{v},$$
where $\tilde{v} = h_*v$, $\tilde{w}=h_*w$. 

Assuming \( v \in \ker A_\eta \), and after applying \( h_* \) to the second equation and subtracting it from the first, and using the relation \( h_* A_{\xi} = A_{\xi} h_* \), along with the invariance of \( \ker A_\eta \) under the isometries of \( M^n \), we obtain
\[
\langle \nabla_v^\perp \eta, \xi \rangle h_* A_\eta w = \langle \nabla_{h_*v}^\perp \eta, \xi \rangle A_{\eta} h_* w
\]
for all \( w \in TM \). Given that \( h_* A_\eta \neq \tilde{A}_{\tilde{\eta}} h_* \) and that \( \det A_\eta \) is constant (see Lemma \ref{lem:Rank2}), we deduce
\[
\langle \nabla_v^\perp \eta, \xi \rangle = 0
\]
for all \( v \in \ker A_\eta \). This concludes the proof of this item.

\vspace{0.1cm}

The proof of item~\textit{(\ref{it:Eigenvalue})} is more involved. First, we must establish the existence of an eigenvalue \(\lambda\) of \(A_\xi\) with eigenspace \(E_{\lambda}\) such that \(\ker A_\eta \subset E_{\lambda}\). To achieve this, we introduce a specific basis that generates the orthogonal complement of \(\ker A_\eta^{\perp}\). Directly from the Codazzi equation, we have
\begin{equation}\label{eq:defZs}
    \langle \nabla_Z^\perp \eta, \xi \rangle \langle A_\xi X, Y \rangle = \langle \nabla_X Y, A_\eta Z \rangle,
\end{equation}
where \(X,Y \in \ker A_\eta\) and \(Z \in \mathrm{Im}\  A_\eta\). Define the one-form \(\omega(\cdot)=\langle \nabla_{\cdot}^\perp \eta, \xi \rangle\) associated with the normal connection and let \(d = \dim \text{span}\{\nabla_X Y : X,Y \in \ker A_\eta\}^\perp\). If \(d=2\) then \(\ker A_\eta\) is an autoparallel distribution, a contradiction with item~\textit{(\ref{it:NotTotGeo})}. If \(d=0\) then \(\rank\ A_\eta \leq 1\) since \(A_\eta (\ker \omega) \subset \text{span}\{\nabla_X Y : X,Y \in \ker A_\eta\}^\perp\) by \eqref{eq:defZs}. We conclude then that the distribution \(\text{span}\{\nabla_X Y : X,Y \in \ker A_\eta\}^\perp\) is one-dimensional. Using equation \eqref{eq:defZs}, we can locally define the orthonormal vector fields \(\{Z_1,Z_2\}\) such that 
\begin{equation}\label{eq:defZ2}
    \langle \nabla_X Y, Z_2 \rangle = 0,
\end{equation}
and 
\begin{equation}\label{eq:defZ1}
    \langle \nabla_X Y, Z_1 \rangle = \langle \nabla_W^\perp \eta, \xi \rangle \langle A_\xi X, Y \rangle,
\end{equation}
for all \(X,Y \in \ker A_\eta\), where \(A_\eta W = Z_1\) for some \(W \in \ker A_\eta^\perp\). In short, \(\operatorname{span}\{Z_2\} = \text{span}\{\nabla_X Y : X,Y \in \ker A_\eta\}^\perp\) and \(Z_1\) generates its orthogonal complement in \(\ker A_\eta^\perp\). Since the vector fields \(Z_1\) and \(Z_2\) are (locally) invariant under isometries of \(M^n\), the left side of \eqref{eq:defZ1} remains invariant. Given that \(\langle A_\xi X, Y \rangle\) is also invariant, we deduce that \(\langle \nabla_W^\perp \eta, \xi \rangle\) is constant.

Using this basis we will show the existence of the eigenvalue \(\lambda\). Due to Lemma~\ref{lem:Rank2}, the distribution \(\ker A_\eta\) is integrable. Let \(N_x^{n-2}\) denote the leaf of this distribution through a point \(x \in M^n\). By considering the restriction of \(f\) to this leaf, we have the isometric immersion
\[
g \coloneqq f|_{N_x}: N_x^{n-2} \rightarrow \mathbb{H}^{n+2},
\]
and it is a straightforward calculation that the second fundamental form of \(g\) is given by 
\begin{equation}\label{eq:secondG}
    \alpha_g(X,Y) = \langle A_\xi X, Y \rangle \left( \xi + \langle\nabla_W^\perp \eta, \xi \rangle Z_1 \right).
\end{equation}
Consequently, the first normal bundle of \(g\) consists of one-dimensional fibers generated by \(\zeta = \xi + \langle\nabla_W^\perp \eta, \xi \rangle Z_1\), which has a constant norm along \(N_x^{n-2}\). Additionally, by \cite[Theorem~A]{sNullitiesReductionDajczerRodriguez} (or \cite[Proposition 2.7]{base}) along with our assumptions regarding the ranks, we have that the first normal bundle of \(g\) is parallel. This means that \(\zeta\) is parallel with respect to the normal connection of \(g\). In particular, we have 
\begin{equation}\label{eq:Z_1Perp}
    0= \langle {}^{\mathbb{H}}\nabla_X\zeta, Z_1 \rangle =  - \langle A_{\xi} X, Z_1 \rangle\ \text{for all }X \in \ker A_{\eta},
\end{equation}
where \({}^\mathbb{H}\nabla\) is the Levi-Civita connection of \(\mathbb{H}^{n+2}\), and a reduction of the codimension \(g: N_x^{n-2} \rightarrow \mathbb{H}^{n-1}\). We can then use the works \cite{TakahashiCod1Hom,TakahashiCod1Hom3}, which classify the homogeneous hypersurfaces of the hyperbolic space. The only possible cases with respect to the shape operator \(A_\zeta^{(g)}\) of \(g\) are
\begin{enumerate}[label=(\roman*)]
    \item \label{item:Case1} \(\rank\ A_{\zeta}^{(g)} \leq 1\);
    \item \label{item:Case2}\(\rank\ A_{\zeta}^{(g)} =n-2\) and we are in one of the following situations:
    \begin{itemize}
        \item \(g\) is umbilical and \(N_x^{n-2}\) is isometric to a sphere, to the hyperbolic space, or to the Euclidean space;
        \item \(g\) is not umbilical and \(N_x^{n-2}\) is isometric to the Riemannian product of a sphere and a hyperbolic space.
    \end{itemize}
\end{enumerate}
If we are in case \ref{item:Case1}, it is straightforward to verify that \(\ker A_\zeta \cap [A_\xi(\Ima A_\eta)]^\perp \subset \Delta\). It follows from \eqref{eq:Z_1Perp} that \(\dim [A_\xi(\Ima A_\eta) \cap \ker A_\eta] \leq 1\), and thus \(\nu \geq n-4\) which is a contradiction since \(n \geq 5\) and the submanifold have trivial relative nullity.
Hence, the only scenarios to consider are those in case \ref{item:Case2}. From item~\textit{(\ref{it:NotTotGeo})}, \(\ker A_\eta\) is neither an autoparallel distribution nor is \(N_x\) a totally geodesic submanifold. By combining Lemma 17 and Lemma 18 from \cite{spaceFormsNoronhadeCastro}, we deduce that \(g\) is umbilical and each leaf of the distribution \(\ker A_\eta\) is isometric to a Euclidean space. Using \eqref{eq:secondG}, we establish that \(\langle A_{\xi} X, Y \rangle = \lambda \langle X, Y \rangle\) for every \(X, Y \in \ker A_\eta\), where \(\lambda\) is a nonzero constant. From \eqref{eq:Z_1Perp} and \cite[Lemma 11]{spaceFormsNoronhadeCastro}, we conclude that \(\ker A_\eta\) is preserved by \(A_\xi\) and that \(\lambda\) is an eigenvalue of \(A_\xi\). In conclusion, \(\ker A_\eta \subset E_{\lambda}\), where \(E_\lambda\) denotes the eigenspace of \(A_\xi\) associated with the eigenvalue \(\lambda \in \R \setminus \{0\}\).

We will now show that \(\text{span}\left\{ \nabla_X Y : X,Y \in \ker A_\eta\right\} = E_\lambda\). Using the Codazzi equation in the direction \(\xi\) with the vectors \(X \in \ker A_\eta\) and \(Z \in (\ker A_\eta)^\perp\), it follows from item~\textit{(\ref{it:ParallelAlong})} that
\begin{equation}
\nabla_X A_\xi Z - A_\xi\nabla_X Z = \lambda \nabla_Z X - A_\xi \nabla_Z X.
\end{equation}
Taking the inner product with \(Y \in \ker A_\eta\), we get
\begin{equation}
\left \langle \nabla_X A_\xi Z, Y \right \rangle - \lambda\left \langle \nabla_X Z, Y \right \rangle = 0,
\end{equation}
which is equivalent to
\begin{equation}
\lambda \left \langle Z, Y \right \rangle X - \left \langle A_\xi Z, \nabla_X Y \right \rangle - \lambda\left \langle \nabla_X Z, Y \right \rangle = 0.
\end{equation}
Hence, \(A_\xi \nabla_X Y = \lambda \nabla_X Y\) and item~\textit{(\ref{it:Eigenvalue})} follows. Observe that since \((\text{span}\{\nabla_X Y : X,Y \in \ker A_\eta\})^\perp = \text{span}\{Z_2\}\), it follows from the definition of the orthonormal frame \(\{Z_1,Z_2\}\) that \(Z_1 \in E_\lambda\) and from the minimality of \(f\) we have \(A_\xi Z_2 = -(n-1)\lambda Z_2\).

\vspace{0.1cm}

In order to prove item~\textit{(\ref{it:PerpTotGeo})}, we need to show that $\nabla_{Z_i} Z_j \in \text{span}\{Z_1,Z_2\}$ for all $1 \leq i,j \leq 2$. The Codazzi equation for the shape operator $A_\xi$ in the directions $Z_2$ and $X \in \ker A_\eta$ gives
\[
\nabla_{Z_2} A_\xi X - A_{\nabla^{\perp}_{Z_2}\xi} X - A_\xi \nabla_{Z_2} X = \nabla_X A_\xi Z_2 - A_{\nabla^{\perp}_X\xi} Z_2 - A_\xi \nabla_X Z_2.
\]
Using that $\ker A_\eta \subset E_\lambda$ and $A_\xi Z_2 = -(n-1)\lambda Z_2$, the previous equation reduces to
\[
\lambda \nabla_{Z_2} X - A_\xi \nabla_{Z_2}X = -n \lambda \nabla_X Z_2.
\]
Taking the inner product with $Z_2$, we have the following $$\left \langle \nabla_{Z_2} Z_2, X \right \rangle = 0.$$ Additionally, taking the inner product with $Z_1$, we obtain $\left \langle \nabla_{X} Z_2, Z_1 \right \rangle = \left \langle \nabla_{X} Z_1, Z_2 \right \rangle = 0$. Analogously, the Codazzi equation for the shape operator $A_\xi$ in the directions $Z_2$ and $X \in \ker A_\eta$ reduces to $A_\xi \nabla_{Z_1} X = \lambda \nabla_{Z_1} X$. In particular, $\nabla_{Z_1} X \in \text{span}\{{Z_2}\}^{\perp} = \text{span}\{\nabla_X Y: X,Y \in \ker A_\eta\}$, which implies \[\left \langle \nabla_{X} Z_1, Z_2 \right \rangle = 0\] for all $X \in \ker A_\eta$.

The Codazzi equation for the shape operator \(A_\eta\) in directions \(X \in \ker A_\eta\) and \(Z \in \ker A_\eta^{\perp}\) gives
\[
\nabla_X A_\eta Z - A_\eta \nabla_X Z = - A_{\nabla_Z^\perp \eta} X - A_\eta \nabla_Z X.\]
Taking the inner product with \(\tilde{W} \in \Ima\  A_\eta\) and using the property that \(A_\xi\) preserves \(\ker A_\eta\), we obtain
\[\left \langle \nabla_X A_\eta Z, \tilde{W}\right \rangle - \left \langle \nabla_X Z, A_\eta \tilde{W}\right \rangle = - \left \langle \nabla_Z X, A_\eta \tilde{W}\right \rangle.\]
For eigenvectors \(\tilde{Z}_1, \tilde{Z}_2\) of \(A_\eta\) with eigenvalues \(\delta_1, \delta_2\), when \(Z = \tilde{W} = \Tilde{Z}_i\) for \(i \in \{1,2\}\) we have the following relation
\begin{equation}\label{eq:eigenvalueeta}
X(\delta_i) = -\delta_i \left \langle \nabla_{\tilde{Z}_i} X, \tilde{Z}_i\right \rangle.
\end{equation}
Using the minimality of \(f\), it follows that \(\delta = \delta_1 = -\delta_2\). Given the rank hypothesis on \(A_\eta\), we have \(\delta_1 \neq 0\). Due to the homogeneity of \(M^n\), and the fact that \(\det A_\eta\) is constant, we conclude
\[X(\delta_i) = 0 \implies \left \langle \nabla_{\tilde{Z}_i} X, \tilde{Z}_i\right \rangle = 0, \quad \forall i \in \{1,2\}.\]
Considering the specific vectors \(Z = \tilde{Z}_1\) and \(W = \tilde{Z}_2\), the Codazzi equation yields
\begin{align*}
2\delta\left \langle \nabla_X \tilde{Z}_1, \tilde{Z}_2\right \rangle &= \delta\left \langle \nabla_{\tilde{Z}_1} X, \tilde{Z}_2\right \rangle, \\
2\delta\left \langle \nabla_X \tilde{Z}_2, \tilde{Z}_1\right \rangle &= \delta\left \langle \nabla_{\tilde{Z}_2} X, \tilde{Z}_1\right \rangle.
\end{align*}
From these, we deduce the following symmetric relation
\[\left \langle X, \nabla_{\tilde{Z}_2} \tilde{Z}_1\right \rangle = -\left \langle X, \nabla_{\tilde{Z}_1} \tilde{Z}_2\right \rangle
\]
for all $X \in \ker A_{\xi}$. Further, representing \(Z_1, Z_2\) as a linear combination of \(\tilde{Z}_1, \tilde{Z}_2\), namely:
\begin{align*}
Z_1 &= \sin\theta \tilde{Z}_1 + \cos\theta\tilde{Z}_2, \\
Z_2 &= \cos\theta \tilde{Z}_1 - \sin\theta\tilde{Z}_2,
\end{align*}
for some angle function \(\theta \in \left(0, \frac{\pi}{2}\right)\). It follows from $$\left \langle \nabla_{Z_2} X , Z_2\right \rangle=\left \langle \nabla_{\tilde{Z}_1} X , \tilde{Z}_1\right \rangle = \left \langle \nabla_{\tilde{Z}_2} X , \tilde{Z}_2\right \rangle=0$$ for all $X \in \ker A_\eta$, that $$0=\left \langle \nabla_{Z_2} Z_2 , X\right \rangle = -\cos\theta\sin\theta \left \langle \nabla_{\tilde{Z}_1} \tilde{Z}_2 + \nabla_{\tilde{Z}_2} \tilde{Z}_1 , X\right \rangle= - \left \langle \nabla_{Z_1} Z_1 , X\right \rangle.$$ It remains to calculate $\left \langle \nabla_{Z_2} Z_1 , X\right \rangle$. It follows from $$0=\left \langle \nabla_{Z_1} Z_2 , X\right \rangle = \left \langle -\sin^2\theta \nabla_{\tilde{Z}_1} \tilde{Z}_2 +\cos^2\theta \nabla_{\tilde{Z}_2} \tilde{Z}_1, X\right \rangle = \left \langle \nabla_{\tilde{Z}_2} \tilde{Z}_1, X\right \rangle$$ that $\left \langle \nabla_{Z_2} Z_1 , X\right \rangle = 0$ and this concludes this item.

\vspace{0.1cm}
For item~\textit{(\ref{it:Spherical})}, we will prove that the distribution \(\ker A_\eta\) is spherical in \(M^n\) with mean curvature vector field \(\lambda\left \langle \nabla^\perp_W \eta, \xi \right \rangle Z_1\). We already know that the distribution \(\ker A_\eta\) is umbilical with this mean vector field, as shown by \eqref{eq:defZ1}. Using that $\lambda\left \langle \nabla^\perp_W \eta, \xi \right \rangle$ is constant, $Z_1$ is an unitary vector field, and $\left \langle \nabla_X Z_1 , Z_2 \right \rangle = 0$ from the proof of item~\textit{(\ref{it:Eigenvalue})} we have that $$\left(\nabla_X\left(\lambda\left \langle \nabla^\perp_W \eta, \xi \right \rangle Z_1\right) \right)\vert_{\ker A_\eta^\perp}=0,$$ which concludes the proof of the proposition.
\end{proof}

\begin{remark}\label{rmk:Z_1}
     As established in Proposition \ref{prop:StructureKerAeta}, the basis \(\{Z_1,Z_2\}\) satisfies the following conditions for all \(X \in \ker A_\eta\):
    \[
    \left \langle \nabla_X Z_1, Z_2 \right \rangle = \left \langle \nabla_{Z_1} Z_1, X \right \rangle = \left \langle \nabla_{Z_2} Z_2, X \right \rangle = \left \langle \nabla_{Z_1} Z_2, X \right \rangle = \left \langle \nabla_{Z_2} Z_1, X \right \rangle = 0.
    \]
    When applying the Codazzi equation in the direction of \(\xi\) for the vectors \(Z_1\) and \(Z_2\), and subsequently taking the inner product with \(Z_1\) and using that $\ker A_\eta^\perp$ is an autoparallel distribution, we obtain
    \[
    \lambda \left \langle \nabla_{Z_1} Z_2, Z_1 \right \rangle =  \left \langle A_\eta (\left\langle \nabla^\perp_{Z_2} \eta, \xi \right\rangle Z_1 - \left\langle \nabla^\perp_{Z_1} \eta, \xi \right\rangle Z_2), Z_1 \right \rangle.
    \]
    It follows from \eqref{eq:defZs} that $\left \langle A_\eta (\left\langle \nabla^\perp_{Z_2} \eta, \xi \right\rangle Z_1 - \left\langle \nabla^\perp_{Z_1} \eta, \xi \right\rangle Z_2), Z_1 \right \rangle=0$, from which it can be deduced that
    \[
    \nabla_{Z_1} Z_1 = 0.
    \]
\end{remark}

\subsection{Proof of the Main Theorem} Assume that $M^n$ is simply connected. If not, we consider the isometric immersion induced by the universal cover of $M^n$, ensuring that the lift of $f$ remains an isometric immersion. If $\nu_0 \neq 0$ then $f$ is totally geodesic by Proposition~\ref{prop:scalBound}. From now on, we assume $\nu_0 = 0$. By Theorem~\ref{thm:HypRank2} either $f$ is rigid, which implies that $f$ is totally geodesic by Theorem~\ref{cor:homScalaOlmos}, or for every point $x\in M^n$ there exist orthonormal vectors $\xi, \eta \in T_f^\perp M^n (x)$ such that $\rank\ A_\eta \leq 2$ and $i_* \circ A_\xi = A_\xi \circ i_*$ for every $i \in \mathrm{Iso} (M^n)$. Then, by Lemma \ref{lem:Rank2}, either $f$ is totally geodesic or $\rank\ A_\eta =2$ everywhere. Henceforth, we will assume that we are in the last case. 

Proposition \ref{prop:StructureKerAeta} implies that the distribution $\ker A_\eta$ is spherical in $M^n$, that $\ker A_\eta^\perp$ is an autoparallel distribution and the shape operators of $f$ are given by

\begin{equation}\label{eq:ShapeOperators}
 A_\xi = 
\left( \begin{array}{@{}c|c@{}}
   \begin{matrix}
    \lambda & 0 \\
    0 & -(n-1)\lambda
   \end{matrix}
      & 0 \\
   \cline{1-2}
   0 & \lambda I \\
\end{array} \right), \quad A_\eta = \left( \begin{array}{@{}c|c@{}}
   \begin{matrix}
      a & b \\
      b & -a 
   \end{matrix}
      & 0 \\
   \cline{1-2}
   0 & 0 \\
\end{array} \right),
\end{equation}

\noindent in the basis $\{Z_1,Z_2,X_1,\dots,X_{n-2}\}$ that diagonalizes $A_\xi$, where $a,b \in C^\infty(M)$ are constant along \(\ker A_\eta\). This basis is specified in item \textit{(\ref{it:Eigenvalue})} of Proposition \ref{prop:StructureKerAeta}. By Lemma \ref{lem:Rank2}, $\det A_\eta|_{\ker A_\eta^\perp}$ is constant. We can assume that the functions $a$ and $b$ are not constant; otherwise, $f$ would be isometrically rigid. Indeed, it follows that any composition of type $g = f \circ i$ for $i \in \mathrm{Iso} (M^n)$ will have the same second fundamental form as $f$ (globally) and, according to \cite{nomizuUniqNormalConn} (see Proposition 4.17 in \cite{base}), would have the same normal connection, implying that $f$ and $g$ are congruent, that is, the original submanifold is extrinsically homogeneous.

The objective now is to show that all this information, along with the fundamental equations of isometric immersions and the compatibility equations of the Levi-Civita connection, contradicts the non-constancy of the functions $a$ and $b$. For this purpose, remember that we have an orthonormal basis $\{Z_1,Z_2\}$ whose span generates an autoparallel distribution such that $\omega = \left \langle \nabla_{Z_2} Z_1, Z_2 \right \rangle$ is constant and $\nabla_{Z_1}Z_1 = 0$ (see Remark \ref{rmk:Z_1}), and we have the two operators $A_\xi, A_\eta$ given by \eqref{eq:ShapeOperators} in this basis. To prove our claim, we will use information from the Gauss and Codazzi equations.
Since $\det A_\eta|_{\ker A_\eta^\perp}$ is constant, we deduce
\[
a^2 + b^2 \equiv r^2,
\]
for some positive constant \( r \in \mathbb{R} \). This leads to the representation of the functions \( a \) and \( b \) in terms of a function \( \theta \), which is constant along \(\ker A_{\eta}\), and the nonzero constant \( r \):

\[
a = r \cos \theta, \quad b = r \sin \theta.
\]

The Codazzi equation in the direction of $\xi$, evaluated on the vectors $Z_1$ and $Z_2$, yields
\[
-\left \langle \nabla^\perp_{Z_1} \xi, \eta \right \rangle A_\eta Z_2 = \lambda \nabla_{Z_2}Z_1 - \left \langle \nabla^\perp_{Z_2} \xi, \eta \right \rangle A_\eta Z_1 + (n-1)\lambda\left \langle \nabla_{Z_2} Z_1, Z_2 \right \rangle Z_2.
\]
Using that $\nabla_{Z_1} Z_1 = 0$ implies $\nabla_{Z_1} Z_2 = 0$. This can be simplified to two equations related to the coefficients in the $Z_1$ direction,
\[
g_2 \cos \theta - g_1 \sin \theta = 0,
\]
and the coefficients in the $Z_2$ direction,
\begin{equation}\label{eq:gConst}
    r(g_1 \cos \theta + g_2 \sin \theta) = n\lambda \omega,
\end{equation}
where $g_1 = \left \langle \nabla^\perp_{Z_1} \xi, \eta \right \rangle$ and $g_2 = \left \langle \nabla^\perp_{Z_2} \xi, \eta \right \rangle$. From the first equation we conclude that
\[
g_1 = g \cos \theta, \quad g_2 = g \sin \theta
\]
for some function $g$. From \eqref{eq:gConst}, it follows that
\[
rg = n\lambda \omega,
\]
which implies $g$ is a constant.

The Codazzi equation in the direction of $\eta$, evaluated on the vectors $Z_1$ and $Z_2$, gives
\[
\nabla_{Z_1}A_\eta Z_2 - g_1(n-1)\lambda Z_2 = \nabla_{Z_2} A_\eta Z_1 + g_2 \lambda Z_1 -\omega A_\eta Z_2,
\]
which is equivalent to the two equations
\begin{equation}\label{eq:functionh1}
    \left(r Z_1(\theta) \right)\cos \theta - \left(-r Z_2(\theta)+g\lambda -2\omega r\right)\sin \theta= 0
\end{equation}
and
\begin{equation}\label{eq:functionh2}
    \left(r Z_1(\theta)\right) \sin \theta + \left(-\lambda g (n-1) -2\omega r -r Z_2(\theta)\right) \cos \theta= 0.
\end{equation}
From \eqref{eq:functionh1}, we find
\begin{equation}\label{eq:Systemh}
    rZ_1(\theta) = h \sin \theta, \quad  -r Z_2(\theta)+g\lambda -2\omega r = h \cos \theta
\end{equation}
for some function $h$, which is constant along \(\ker A_{\eta}\). Using \eqref{eq:functionh2}, we deduce
\[
h = \lambda g n \cos \theta.
\]
Therefore, \eqref{eq:Systemh} reduces to

\begin{equation}
    \begin{cases}
     Z_1 (\theta) = C \cos \theta \sin \theta\\
     Z_2 (\theta) = \frac{C}{n} -2\omega - C \cos^2\theta
    \end{cases}\,,
\end{equation} where $C \coloneqq \frac{1}{r} gn\lambda \neq 0$. It is straightforward to see that \[[Z_1,Z_2] = -\omega Z_2.\] Using the compatibility of the Levi-Civita connection we have

\begin{align*}
& Z_1(Z_2(\theta)) - Z_2(Z_1(\theta)) - [Z_1,Z_2](\theta) = 0 \\
\Leftrightarrow & \quad C^2\sin^2\theta \cos^2\theta + \frac{C^2}{n}\sin^2\theta - 2C\omega \sin^2\theta - \frac{C^2}{n}\cos^2\theta + C\omega \cos^2\theta + \frac{\omega C}{n} \\
& \quad - 2\omega^2 + C^2\cos^4\theta = 0 \\
\Leftrightarrow & \quad \left(C^2 - \frac{2C^2}{n} + 3C\omega \right) \cos^2\theta + \frac{C^2}{n} - 2C\omega + \frac{C\omega}{n} - 2\omega^2 = 0.
\end{align*}

Since $\theta$ cannot be constant, we have the following system:

\begin{equation*}
    \begin{cases}
     C(\frac{n-2}{n} C + 3\omega) = 0,\\
     \frac{C^2}{n} + \frac{1-2n}{n} C\omega-2\omega^2 =0.
    \end{cases}
\end{equation*} Substituting $\omega = \frac{2-n}{3n}C$ into the second equation of the system leads to 

\begin{equation*}
1 + \frac{(1-2n)(2-n)}{3n} -2 \frac{4-4n+n^2}{9n} = 0 
\quad \Leftrightarrow  \quad 2n^2+n-1 = 0.
\end{equation*} Therefore, as such equation cannot be satisfied for any $n \geq 5$, the contradiction arises and it follows that the function $\theta$ is constant, meaning the functions $a$ and $b$ must be constant, thereby concluding the proof of the theorem. $\hfill\blacksquare$

\section*{Acknowledgements}
F. Guimar\~aes is supported by the Para\'iba State Research Support Foundation (FAPESQ/PB) and partially by the Brazilian National Council for Scientific and Technological Development (CNPq), grant 409513/2023-7. F. Guimarães also acknowledges the time spent at the Geometry Section of KU Leuven, which was supported by the Research Foundation-Flanders (FWO) and the Fonds de la Recherche Scientifique (FNRS) under EOS Project G0H4518N. J. Van der Veken is supported by the Research Foundation - Flanders (FWO) and the Fonds de la Recherche Scientifique (FNRS) under EOS project G0I2222N and by the KU Leuven Research Fund under project 3E210539.

\bibliographystyle{abbrv}
\bibliography{bibliography}

\vskip 0.2cm

\noindent Felippe Guimarães

\noindent Departamento de Matemática, \\
Universidade Federal da Paraíba, \\
Cidade Universitária, s/n - Castelo Branco, \\
João Pessoa, PB, 58051-900, Brazil

\noindent e-mail: {\tt fsg@academico.ufpb.br}

\vskip 0.2cm 

\noindent Joeri Van der Veken

\noindent Department of Mathematics, \\
KU Leuven, \\
Celestijnenlaan 200B - Box 2400, \\
3001 Leuven, Belgium 

\noindent e-mail: {\tt joeri.vanderveken@kuleuven.be}

\end{document}